\documentclass[11pt]{article}

\usepackage{amsmath}
\usepackage{amsthm}
\usepackage{amssymb}
\usepackage{latexsym}
\usepackage{mathabx}
\usepackage{pstricks}
\usepackage{graphicx, pst-plot, pst-node, pst-text, pst-tree}
\usepackage{titlefoot}
\usepackage{enumerate}
\usepackage{titlesec}
\usepackage{units}
\usepackage[small,it]{caption}

\usepackage{color}
\definecolor{lightgray}{rgb}{0.9, 0.9, 0.9}
\definecolor{darkgray}{rgb}{0.7, 0.7, 0.7}
\definecolor{darkblue}{rgb}{0, 0, .4}

\usepackage[bookmarks]{hyperref}
\hypersetup{}

\newtheorem{theorem}{Theorem}
\newtheorem{proposition}[theorem]{Proposition}
\newtheorem{lemma}[theorem]{Lemma}

\newcommand{\bprop}{\begin{proposition}}
\newcommand{\eprop}{\end{proposition}}

\newcommand{\minisec}[1]{\bigskip\noindent{\bf #1.}}

\makeatletter
\newfont{\footsc}{cmcsc10 at 8truept}
\newfont{\footbf}{cmbx10 at 8truept}
\newfont{\footrm}{cmr10 at 10truept}
\renewenvironment{abstract}%
		{
		  \begin{list}{}%
		     {\setlength{\rightmargin}{1in}%
		      \setlength{\leftmargin}{1in}}%
		   \item[]\ignorespaces\begin{small}}%
		 {\end{small}\unskip\end{list}}

\amssubj{05C75, 06A07}

\newpagestyle{main}[\small]{
        \headrule
        \sethead[\usepage][][]
        {\sc A Ramsey Theorem for Indecomposable Matchings}{}{\usepage}}

\title{\sc A Ramsey Theorem for Indecomposable Matchings}
\author{%
James Fairbanks\\
\small University of Florida\\[-0.4ex]
\small Gainesville, Florida USA\\[-1.5ex]
}

\titleformat{\section}
        {\large\sc}
        {\thesection.}{1em}{}   

\date{}

\begin{document}
\maketitle

\pagestyle{main}

\begin{abstract}
A matching is indecomposable if it does not contain a nontrivial contiguous segment of vertices whose neighbors are entirely contained in the segment.  We prove a Ramsey-like result for indecomposable matchings, showing that every sufficiently long indecomposable matching contains a long indecomposable matching of one of three types: interleavings, broken nestings, and proper pin sequences.
\end{abstract}

\section{Introduction}

A (labeled, complete) \emph{matching} is a graph on the vertex set
$[2n]=\{1,2,\dots,2n\}$ in which every vertex is incident to exactly one
edge.  An \emph{interval} in a matching is a contiguous segment
of vertices $[i,j]=\{i,i+1,\dots,j\}$ such that no vertex in $[i,j]$
is adjacent to a vertex outside $[i,j]$.  Every matching
has two \emph{trivial} intervals: the empty set and the set of all its
vertices.  A  matching is said to be \emph{indecomposable} if
it has no other intervals (and \emph{decomposable} if it does have nontrivial intervals, see Figure~1).  We prove a Ramsey-like result for
indecomposable matchings, showing that every such matching
contains a large member of one of three explicit families of
indecomposable matchings.

%This figure shows a decomposable matching and labels its non trivial interval
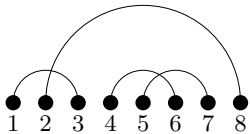
\begin{figure}[t]
\begin{footnotesize}
\begin{center}
\begin{tabular}{c}
	\psset{xunit=0.17in, yunit=0.17in}
	\psset{linewidth=0.005in}
	\begin{pspicture}(1,0)(8,4)
		\pscircle*(1,1){0.04in}\uput[270](1,1){$1$}
		\pscircle*(2,1){0.04in}\uput[270](2,1){$2$}
		\pscircle*(3,1){0.04in}\uput[270](3,1){$3$}
		\pscircle*(4,1){0.04in}\uput[270](4,1){$4$}
		\pscircle*(5,1){0.04in}\uput[270](5,1){$5$}
		\pscircle*(6,1){0.04in}\uput[270](6,1){$6$}
		\pscircle*(7,1){0.04in}\uput[270](7,1){$7$}
		\pscircle*(8,1){0.04in}\uput[270](8,1){$8$}
		\psarc(2,1){1\psxunit}{0}{180}
		\psarc(5,1){3\psxunit}{0}{180}
		\psarc(5,1){1\psxunit}{0}{180}
		\psarc(6,1){1\psxunit}{0}{180}
	\end{pspicture}
\\[8pt]
\end{tabular}
\caption{A decomposable matching with the nontrivial interval $[4,7]$}
\label{fig-dec-matchings}
\end{center}
\end{footnotesize}
\end{figure}

Indecomposable matchings have been studied by Nijenhuis and 
Wilf~\cite{nijenhuis:the-enumeration:}, who provided a recursive
algorithm for constructing all indecomposable matchings.
From their recursion, it follows that the number, $s_n$, of 
indecomposable matchings of $[2n]$ satisfies the recurrence

$$
s_n=(n-1)\sum_{i=1}^{n-1}s_is_{n-k}.
$$

The contribution of this paper is to show that there are, essentially, only three types of indecomposable matchings.

\begin{theorem}
\label{thm-ind-match-ramsey} 
Every indecomposable matching with at least $(2k)^{2k}$ edges contains
a broken nesting, interleaving or proper pin sequence with k edges.
\end{theorem}

This result is the matching analogue of the results of Brignall, Huczynska, and Vatter~\cite{brignall:decomposing-sim:}, who proved a similar result for permutations.

In Theorem~\ref{thm-ind-match-ramsey}, we say that the matching $M$ contains the matching $N$ if $N$ can be obtained from $M$ by deleting a collection of edges and the vertices incident with those edges, and then relabeling the remaining vertices.  For the remainder of this section we discuss the three types of indecomposable matchings mentioned in Theorem~\ref{thm-ind-match-ramsey}.  The proof of the theorem follows in the next section.

%Interleaving
% Main Figure
\begin{figure}
\begin{footnotesize}
\begin{center}
\begin{tabular}{ccccccc}
	\psset{xunit=0.17in, yunit=0.17in}
	\psset{linewidth=0.005in}
	\begin{pspicture}(1,0)(8,4)
		\pscircle*(1,1){0.04in}\uput[270](1,1){$1$}
		\pscircle*(2,1){0.04in}\uput[270](2,1){$2$}
		\pscircle*(3,1){0.04in}\uput[270](3,1){$3$}
		\pscircle*(4,1){0.04in}\uput[270](4,1){$4$}
		\pscircle*(5,1){0.04in}\uput[270](5,1){$5$}
		\pscircle*(6,1){0.04in}\uput[270](6,1){$6$}
		\pscircle*(7,1){0.04in}\uput[270](7,1){$7$}
		\pscircle*(8,1){0.04in}\uput[270](8,1){$8$}
		\psarc(3,1){2\psxunit}{0}{180}
		\psarc(4,1){2\psxunit}{0}{180}
		\psarc(5,1){2\psxunit}{0}{180}
		\psarc(6,1){2\psxunit}{0}{180}
	\end{pspicture}
&\rule{0pt}{0.5in}&
	\psset{xunit=0.17in, yunit=0.17in}
	\psset{linewidth=0.005in}
	\begin{pspicture}(1,0)(8,4)
		\pscircle*(1,1){0.04in}\uput[270](1,1){$1$}
		\pscircle*(2,1){0.04in}\uput[270](2,1){$2$}
		\pscircle*(3,1){0.04in}\uput[270](3,1){$3$}
		\pscircle*(4,1){0.04in}\uput[270](4,1){$4$}
		\pscircle*(5,1){0.04in}\uput[270](5,1){$5$}
		\pscircle*(6,1){0.04in}\uput[270](6,1){$6$}
		\pscircle*(7,1){0.04in}\uput[270](7,1){$7$}
		\pscircle*(8,1){0.04in}\uput[270](8,1){$8$}
		\psarc(4,1){3\psxunit}{0}{180}
		\psarc(4,1){2\psxunit}{0}{180}
		\psarc(4,1){1\psxunit}{0}{180}
		\psarc(6,1){2\psxunit}{0}{180}
	\end{pspicture}
&\rule{0pt}{0.5in}&
	\psset{xunit=0.17in, yunit=0.17in}
	\psset{linewidth=0.005in}
	\begin{pspicture}(1,0)(8,4)
		\pscircle*(1,1){0.04in}\uput[270](1,1){$1$}
		\pscircle*(2,1){0.04in}\uput[270](2,1){$2$}
		\pscircle*(3,1){0.04in}\uput[270](3,1){$3$}
		\pscircle*(4,1){0.04in}\uput[270](4,1){$4$}
		\pscircle*(5,1){0.04in}\uput[270](5,1){$5$}
		\pscircle*(6,1){0.04in}\uput[270](6,1){$6$}
		\pscircle*(7,1){0.04in}\uput[270](7,1){$7$}
		\pscircle*(8,1){0.04in}\uput[270](8,1){$8$}
		\psarc(5,1){3\psxunit}{0}{180}
		\psarc(5,1){2\psxunit}{0}{180}
		\psarc(5,1){1\psxunit}{0}{180}
		\psarc(3,1){2\psxunit}{0}{180}
	\end{pspicture}
&\rule{0pt}{0.5in}&
	\psset{xunit=0.17in, yunit=0.17in}
	\psset{linewidth=0.005in}
	\begin{pspicture}(1,0)(8,4)
		\pscircle*(1,1){0.04in}\uput[270](1,1){$1$}
		\pscircle*(2,1){0.04in}\uput[270](2,1){$2$}
		\pscircle*(3,1){0.04in}\uput[270](3,1){$3$}
		\pscircle*(4,1){0.04in}\uput[270](4,1){$4$}
		\pscircle*(5,1){0.04in}\uput[270](5,1){$5$}
		\pscircle*(6,1){0.04in}\uput[270](6,1){$6$}
		\pscircle*(7,1){0.04in}\uput[270](7,1){$7$}
		\pscircle*(8,1){0.04in}\uput[270](8,1){$8$}
		\psarc(4,1){1\psxunit}{0}{180}
		\psarc(5.5,1){1.5\psxunit}{0}{180}
		\psarc(3.5,1){2.5\psxunit}{0}{180}
		\psarc(5,1){3\psxunit}{0}{180}
		% label positions are calculated by extending the circle 0.4 xunits.
		\rput(3.5, 2.3){$p_1$}
		\rput(6.6, 2.7){$p_2$}
                   \rput(1.1, 2.7){$p_3$}
		\rput(8.2, 2.7){$p_4$}
	\end{pspicture}
\end{tabular}
\end{center}
\end{footnotesize}
\caption{From left to right, the interleaving on $[8]$, the right-broken nesting on $[8]$, the left-broken nesting on $[8]$, and a proper pin sequence on $[8]$.}
\label{fig-ind-matchings}
\end{figure}
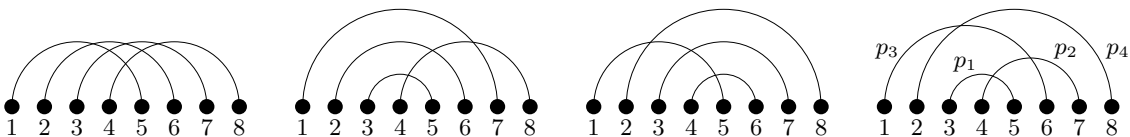

The \emph{interleaving} on $[2n]$ is the indecomposable matching defined by $i\sim i+n$ for all $i \in[n]$.  The interleaving on $[8]$ is depicted in the first matching of Figure~2.

%Broken Nesting.

The \emph{nesting} on $[2n-2]$ is the matching defined by $i\sim
 2n-2-i+1$ for $i\in [n]$.  This matching is not indecomposable, but can be made indecomposable by adding a new edge which \emph{breaks} the nesting.  This new edge can break the nesting either to the left or the right.  The \emph{right-broken nesting} on $[2n]$ has edges $n\sim 2n$ and $i\sim 2n-i$ for $i\in[n-1]$, while the \emph{left-broken nesting} on $[2n]$ has edges $1\sim n+1$ and $i+1\sim 2n-i+1$ for $i\in[n-1].$ The right- and left-broken nestings on $[8]$ are depicted in the second and third matchings of Figure~\ref{fig-ind-matchings}.

%Pin Sequence

The most diverse family of indecomposable matchings is the family of \emph{pin sequences}.  In order to define pin sequences, we need a few preliminaries.  Given a set of edges in a matching, its \emph{shadow} is the smallest contiguous segment of vertices containing their endpoints.  In an indecomposable matching, every nonempty shadow must either consist of all the vertices, or be \emph{split}, meaning that there is a vertex in the shadow which is adjacent to a vertex outside of the shadow. We refer to such edges as \emph{pins}.

A pin sequence is then a sequence of edges $p_1,p_2,\dots$ such that each $p_i$ breaks the shadow of $\{p_1,\dots,p_{i-1}\}$.   Thus one pin sequences on $[8]$ is $3\sim 5, 4\sim7, 6\sim 1, 2\sim 8$, shown in the fourth matching of Figure~\ref{fig-ind-matchings}.

%%Proof that all pin squences are indecomposable
It is important to verify that all pin sequences are indecomposable.
\begin{proposition} \label{pinSeqsAreIndecomp}
  Every pin sequence is an indecomposable matching.
\end{proposition}
\begin{proof}
Let $P=p_1, p_2, \dots , p_n $ be a
  pin sequence and suppose  that $P$ has an interval. There must be at least one pin in
  any nontrivial interval and since the $P$ is finite there
  is a largest $i$ such that $p_i$ is in the interval, this is the
  last pin in the interval. Suppose $i\neq 2n$. Because $P$ is a pin
  sequence we know that $p_{i+1}$ crosses $p_i$ which violates our
  assertion that $p_{i}$ was the last pin in the interval so $p_i$
  must have been $p_{2n}$.
%%pin sequences are not pin sequences backwards but we don't need that because we just need the symmetry of crossing.
We show that every interval contains $p_1$. Suppose $p_i \neq p_1$ is the first pin an interval. Since $P$ is a pin sequence then $p_i$ crosses $p_{i-1}$ so $p_{i-1}$ is in the interval, but this contradicts the minimality of $i$. So we know that $p_1$ is in every interval.

This shows pin sequences contain only trivial intervals and thus are indecomposable.
  
  \end{proof}

In the statement of Theorem~\ref{thm-ind-match-ramsey} we use a particular type of pin sequence. A \emph{proper} pin sequence satisfies, for each $1<i<2n$, $p_{i+1}$ splits the shadow cast by $\{p_1,\dots, p_{i}\}$ but not the shadow cast by $\{p_1,\dots, p_{i-1}\}$

\section{Proof of Theorem~\ref{thm-ind-match-ramsey}}

Our proof of Theorem~\ref{thm-ind-match-ramsey} consists of analyzing two
possibilities.  First, we show that if a single edge is crossed by many
different edges, then the matching contains an interleaving or broken nesting.
The alternative is that no edge is crossed by many different edges, in which
case we show that the matching contains a long proper pin sequence.

\begin{lemma}
\label{conval}
  If a single edge $e$ is crossed by $2(k-1)^2 + 2$ edges of a matching,
then the matching contains either a broken nesting or an interleaving of length $k$.
\end{lemma}

\begin{proof}%of lemma on convalescence implies broken nesting or interleaving

Every edge that crosses $e$ crosses either to the left or to the right, thus at least $(k-1)^2+1$ of the edges must cross to the same side of $e$. By symmetry call that side left. 
Now order these $(k-1)^2+1$ edges by their left endpoints, preserving the natural order on the integers. Let $S$ be the unique sequence formed by the right vertices of the the edges when read in order.

%Now construct the unique sequence formed by taking the left endpoints of the $(k-1)^{2}+1$ edges in increasing order and reading the labels off their neighbors. 

By the Erd\H{o}s-Szekeres theorem, $S$ has a monotone subsequence of length $k$. If this subsequence is increasing, the matching contains an interleaving. Otherwise this subsequence is
decreasing and the matching contains a nesting that is broken by $e$.

\end{proof} 

In order to prove the main theorem, we will need pin sequences that are \emph{right-reaching}, that is, their final pin is incident with the vertex $2n$.

%%%% we show that we have plenty of pin sequences because we are going to pigeon hole the pin sequence tree

It is helpful to know that proper right-reaching pin sequences are always
available in indecomposable matchings.
\begin{lemma} \label{RRPPS}
Every indecomposable matching has a proper right-reaching pin sequence beginning with any edge.
\end{lemma}

\begin{proof}%lemma RRPPS
Let $p_1$ be an arbitrary edge of the indecomposable matching $M$.  If the vertex $2n$ is incident with $p_1$, then we are done.  Otherwise, by the indecomposability of $M$, there is an edge which crosses $p_1$; label this edge $p_2$.  If $2n$ is incident with $p_2$, then we are done.  Otherwise, the edges $p_1$ and $p_2$ define a new shadow, which is still not an interval, so there is an edge, $p_3$, which splits this shadow.  Since the only interval is $[2n]$, by repeating this process, we can create a pin sequence $p_1,\dots,p_m$ such that $2n$ is incident with $p_m$.

We now construct from this right-reaching pin sequence a proper right-reaching pin sequence $q_1,q_2,\dots,q_s$.  First, set $q_1=p_1$.  Then we successively extend this sequence by choosing $q_i$ to be the pin $p_j$ of the greatest index which crosses $q_{i-1}$.  We stop when $q_i$ is incident with $2n$.  Note that by this selection procedure, $q_i$ crosses $q_{i-1}$ but does not cross $q_1,\dots,q_{i-2}$.  Therefore the resulting sequence $q_1,\dots$ is a proper right-reaching pin sequence, as desired.
\end{proof}

In the proof of Theorem~\ref{thm-ind-match-ramsey}, we use Lemma~\ref{RRPPS} to show that every indecomposable matching with $n$ edges contains at least $n$ distinct right-reaching proper pin sequences.

We can now derive the main result.

\newenvironment{proof-main}{\medskip\noindent\emph{Proof of Theorem~\ref{thm-ind-match-ramsey}. }}
{\qed\bigskip}
\begin{proof-main}
Let $M$ be a matching which does not contain a broken nesting, interleaving, or proper pin sequence with at least $k$ edges.  We construct a tree of all the proper right-reaching pin sequences of $M$ in the following manner.  The parent of the pin sequence $p_1,\dots,p_m$ ($m\ge 2$) is the sequence $p_2,\dots,p_m$, so the root of this tree is the edge (thought of as a pin sequence) incident with the vertex of the greatest label.

Since $M$ does not have a pin sequence of length $k$, this tree has height at most $k-1$.  Because $M$ does not contain an interleaving or broken nesting of with $k$ edges, Lemma~\ref{conval} implies that no node may have $2(k-1)^2+2$ children.  This bounds the size of the tree with the sum
$$
\sum_{i=0}^{k-1} (2(k-1)^2+1)^i
=
\frac{(2k^2-4k+3)^k-1}{2(k-1)^2}
<
\sum_{i=0}^{k-1} (2k^2)^i
=
(2k^2)^{2k}
$$ 
By Lemma~\ref{RRPPS}, every edge of $M$ begins a proper right-reaching pin sequence.  Therefore $M$ can have at most $(2k^2)^{2k}$ edges, proving the theorem.
\end{proof-main}

\minisec{Acknowledgments}
I would like to thank Professor Vince Vatter for guiding me in this research.

%Bibliography
\def\cprime{$'$}

%\bibliographystyle{acm}
%\bibliography{../refs}

\end{document}